\documentclass{article}
\usepackage{amssymb}
\usepackage{amsfonts}
\usepackage{amsmath}

\newtheorem{theorem}{Theorem}

\newtheorem{claim}[theorem]{Claim}

\newtheorem{lemma}[theorem]{Lemma}

\newenvironment{proof}[1][Proof]{\noindent\textbf{#1.} }{\ \rule{0.5em}{0.5em}}
\input{tcilatex}
\begin{document}

\title{ON THE SOLUTIONS OF THE EQUATION $x^{2}+19^{m}=y^{n}$}
\author{$^{1}$Bilge PEKER \\
$^{1}$Elementary Mathematics Education, \\
Ahmet Kelesoglu Education Faculty, \\
Selcuk University, Konya, Turkey, bilge.peker@yahoo.com \and $^{2}$Selin
(INAG) CENBERCI \\
$^{2}$Department of Mathematics,\\
Ahmet Kelesoglu Education Faculty, \\
Selcuk University, Konya, Turkey, \\
inag\_s@hotmail.com}
\maketitle

\begin{abstract}
In this article, we consider the equation $x^{2}+19^{m}=y^{n}$ , $n>2,\ m>0$.

We find the solutions of the title equation for not only $2\mid m$ but also $%
2\nmid m$.
\end{abstract}

\textbf{1. Introduction}

\bigskip

\ \ \ \ \ \ \ The Diophantine equation which is so-called generalized
Ramanujan-Nagell equation \ $x^{2}+k=y^{n},\ k,x,y,n\in 
\mathbb{Z}
$, $n>2$ has been studied extensively. When\ $\ n=3$ \ the equation
clarifies an elliptic curve and it is well known as Mordell's equation.
Mordell studied this type of equation in detail in his book $[12]$. When $%
n>3 $, the equation clarifies a hyperelliptic curve. This case also has a
lot of literature. J.H.E. Cohn $[8]$ solved the equation for 77 values of
positive $k$ under 100. We are interested in the case when $k=c^{m}$, $c$
is\ a positive integer,\ $m\in 
\mathbb{N}
$ \ is\ unknown. There are lots of studies for different cases of $c.$ For
example L. Tao $[13]$\ considered the equation for $c=3\ $and $5.\ $

\bigskip

In $\left[ 6\right] ,$ Cohn considered the title equation for the case $m=1.$
He proved that the equation $x^{2}+19=y^{n}$ has two primitive integer
solutions with $n\geq 3$, namely the solutions are $\left( x,y,k,n\right)
=\left( 18,7,0,3\right) ,\ \left( 22434,55,0,5\right) .\ $On the other hand,
in $\left[ 2\right] $ Arif and Muriefah gave a theorem for the solutions to
the equation $x^{2}+q^{2k+1}=y^{n}$ where $k\geq 0$ and $n\geq 5$ is an odd
integer. With these conditions the equation $x^{2}+q^{2k+1}=y^{n}$ has
exactly two families of solutions given by $\left( x,y,k,n\right) =\left(
22434.19^{5M},55.19^{2M},5M,5\right) $, $\left(
2759646.341^{5M},377.341^{2M},341M,5\right) .$

\bigskip

\ \.{I}.N. Cang\"{u}l \ et. al. $[5]$ found all solutions of the equation \ $%
x^{2}+11^{2k}=y^{n},$ $\ x\geq 1,$ $\ y\geq 1,$ $\ k\in 
\mathbb{N}
,$ $\ n\geq 3.$ E. Demirpolat et. al. $[9]$ solved the equation $%
x^{2}+11^{2k+1}=y^{n}$ .

H. Zhu and M. Le $\left[ 14\right] $ give all solutions of some generalized
Lebesque- Nagell equations \ $x^{2}+q^{m}=y^{n}$, where the class number of
the imaginary quadratic field $%
\mathbb{Q}
\left( \sqrt{-q}\right) $ is one by using known results and elementary
arguments.

\bigskip The aim of this paper is to study the equation $x^{2}+19^{m}=y^{n},$
$n>2$ and $m>0$. We treat the equation for $m$ is an even and odd
separately.\bigskip\ For the following process we need the below Lemma.

\bigskip 

\begin{lemma}
The equation $19x^{2}+1=y^{n},\ n>2$ \ has no positive integer solution$.$
\end{lemma}

\begin{proof}
Suppose $\left( x,y,n\right) $ \ is a positive integer solution. Since $\ n>2
$ arguing $\mathit{modulo}$ 8 one obtains that if \ there exist integers $x,y
$ such that $19x^{2}+1=y^{n}$, then $y$ is an odd and $x$ is an even.

Now for the proof we evaluate two different cases.

$\left( i\right) $ The first case is with $n=4,$ the equation $\
19x^{2}+1=y^{4}.$ Cohn  $[7]$ \ showed that this equation has no positive
integer solution.

$\left( ii\right) $ The second case is with $n=p$ is an odd prime, the
equation $\ 19x^{2}+1=y^{p}$. Now we search this equation whether has a
positive integer solution or not. We suppose that $\left( 1-\sqrt{-19}%
x\right) $ $\ \ $ and $\ \ \left( 1+\sqrt{-19}x\right) $ \ two ideals of the
ring of integers $%
\mathbb{Z}
\left[ \sqrt{-19}\right] $ \ of the field $F=%
\mathbb{Q}
\left[ \sqrt{-19}\right] $ \ are relatively prime. From the unique
factorization of prime ideals in \ $%
\mathbb{Z}
\left[ \sqrt{-19}\right] $, we have $\left( 1-\sqrt{-19}x\right) =\wp ^{p}$
\ for some ideal $\wp $ \ of \ $%
\mathbb{Z}
\left[ \sqrt{-19}\right] .$ \ $\wp $ \ is a principal so there exists
integers $a,b$ such that $\wp =\left( \dfrac{a+b\sqrt{-19}}{2}\right) .$
Hence

$\ \ \ \ \ \ \ \ \ \ \ \ \ \ \ \ \ \ \ \ \ \ \ \ \ \ \ \ \ $%
\begin{equation*}
\ \ \ 1+\sqrt{-19}x=u\left( \dfrac{a+b\sqrt{-19}}{2}\right) ^{p}
\end{equation*}

where $u$ is a unit in $%
\mathbb{Z}
\left[ \sqrt{-19}\right] .$\ Since the units in $%
\mathbb{Z}
\left[ \sqrt{-19}\right] $ \ are just $\pm 1,\ $where $a\neq 0(\func{mod}19)$
and $a$ is an even, then $b$ is an even too since $\wp =\left( \dfrac{a+b%
\sqrt{-19}}{2}\right) .$ So we write $\ a=2A,\ b=2B$ \ .\ We get without
loss of generality

$\ $%
\begin{equation*}
\ \left( 1+\sqrt{-19}x\right) =\left( A+B\sqrt{-19}\right) ^{p}\text{ \ \ \
\ \ \ \ \ \ }(1.1).
\end{equation*}

If we take the norm of both sides then we get $1+19x^{2}=(A^{2}+19B^{2})^{p}$%
.$\ $Since $y^{p}=1+19x^{2}\ $we have $y=A^{2}+19B^{2}.\ $Now comparing the
real parts of the equality (1.1), then we obtain

\begin{equation*}
1=A\dsum\limits_{k=0}^{\frac{p-1}{2}}\left( 
\begin{array}{c}
p \\ 
2k%
\end{array}%
\right) A^{p-\left( 2k+1\right) }\left( -19B^{2}\right) ^{k}.
\end{equation*}

So \ $A=\pm 1.$ We get $y=A^{2}+19B^{2}=1+19B^{2}.$ Since $2\nmid y,$ we
have $2\mid B.$ Therefore $x\geq 1,$ from $19x^{2}+1=y^{p}$ \ we have $y>1$.
Hence from $y=1+19B^{2}$, \ we deduce $B\neq 0.$ On the other hand

\begin{equation*}
1=A\dsum\limits_{k=0}^{\frac{p-1}{2}}\left( 
\begin{array}{c}
p \\ 
2k%
\end{array}%
\right) A^{p-\left( 2k+1\right) }\left( -19B^{2}\right) ^{k}\equiv
A^{p}\left( \func{mod}19\right) .
\end{equation*}

Since $\ A=1$, we obtain

\begin{equation*}
\ 0=\dsum\limits_{k=1}^{\frac{p-1}{2}}\left( 
\begin{array}{c}
p \\ 
2k%
\end{array}%
\right) \left( -19B^{2}\right) ^{k}.
\end{equation*}

Let $V_{2}\left( .\right) $ \ be the standart 2-adic valuation. For $k\geq
2, $ let $\ \ k=2^{s}t,$ $\ 2\nmid t.$ Then when $s=0$ we have $k=t\geq 2$ \
\ and $\ V_{2}\left( k\right) =s=0.$ \ So \ $2\left( k-1\right) =2\left(
t-1\right) \geq 2>0=$ $\ V_{2}\left( k\right) $ \ and when $s>0,2\left(
k-1\right) =2\left( 2^{s}t-1\right) \geq 2\left( 2^{s}-1\right) \geq
2s>s=V_{2}\left( k\right) .$ Since $2\mid B$ \ and $B\neq 0$ \ for $1<k\leq 
\frac{p-1}{2}$ \ we have

$\ V_{2}\left( \left( 
\begin{array}{c}
p \\ 
2k%
\end{array}%
\right) \left( -19B^{2}\right) ^{k}\right) =V_{2}\left( \dfrac{p\left(
p-1\right) }{2k\left( 2k-1\right) }\left( 
\begin{array}{c}
p-2 \\ 
2k-2%
\end{array}%
\right) \left( -19B^{2}\right) ^{k}\right) $

$\geq V_{2}\left( \left( 
\begin{array}{c}
p \\ 
2%
\end{array}%
\right) \left( -19B^{2}\right) \right) +2\left( k-1\right) -V_{2}\left(
k\right) $

$>V_{2}\left( \left( 
\begin{array}{c}
p \\ 
2%
\end{array}%
\right) \left( -19B^{2}\right) \right) $

But from $0=\dsum\limits_{k=1}^{\frac{p-1}{2}}\left( 
\begin{array}{c}
p \\ 
2k%
\end{array}%
\right) \left( -19B^{2}\right) ^{k},$ we see that there are at least two
terms with the smallest 2-adic valuation. This is a contradiction.

$\ \ \ \ \ \ \ \ \ \ \ \ \ \ \ \ $\ \ \ \ \ \ \ \ \ \ \ \ \ \ \ \ \ \ \ \ \
\ \ \ \ \ \ \ \ \ \ \ $\ \ \ \ \ \ \ \ \ \ \ \ \ \ \ \ \ \ \ \ \ \ $
\end{proof}

$\ \ $\textbf{2. The equation }$x^{2}+19^{2k}=y^{n}$

\bigskip 

\bigskip\ Now, we can give our main theorem for $m$ is an even;

\begin{theorem}
The equation $x^{2}+19^{2k}=y^{n},$ with $n>2$ and $k>0$ \ has no positive
integer solution.
\end{theorem}

\begin{proof}
For proof, we must consider three different cases.
\end{proof}

\textbf{The first case is, for }$n\geq 4$\textbf{\ is an even.}

\begin{lemma}
\bigskip If $3\mid k$ or \ $q\equiv \pm 3\ \left( \func{mod}8\right) $ \
then the Diophantine equation $x^{2}+q^{2k}=y^{n}$ \ where $n\geq 4$ is an
even, $q$ is an odd prime and $\left( q,x\right) =1,$ has no solution $[1].$
\end{lemma}

From this Lemma we can say when $n\geq 4$ \ is an even, the equation $%
x^{2}+19^{2k}=y^{n}$ \ has no positive integer solution.

\bigskip

\ \ \textbf{The second case is, }$n=3$\textbf{. }

\bigskip

\begin{claim}
The equation $x^{2}+19^{2k}=y^{3},\ k>0$ \ has no positive integer solution.
\end{claim}

\bigskip

Firstly, we assume that $\left( x,19\right) =1$. Since $x$ \ and \ $y$ are
coprime and $\left( x,19\right) =1,$ then we get $x$ is an even. By
factorization, the equation becomes

$\ \ \ \ \ \ \ \ \ \ \ \ \ \ \ \ \ \ \ \ \ \ \ \ \ \ \ \ $%
\begin{equation*}
\ \ \left( x+19^{k}i\right) \left( x-19^{k}i\right) =y^{3}.
\end{equation*}

$\left( x+19^{k}i\right) $ $\ \ $and $\ \left( x-19^{k}i\right) $ \ are
coprime in $%
\mathbb{Z}
\lbrack i]$. \ Because of the units of $%
\mathbb{Z}
\lbrack i]$ \ are only$\ \pm 1,\ \pm i$ ,\ we get\ \ \ \ \ \ \ \ \ \ \ \ \ \
\ \ \ \ \ \ \ \ \ \ \ \ 

$\ \ \ \ \ \ \ \ \ \ \ \ \ \ \ \ \ \ \ \ \ \ \ \ \ \ \ \ \ \ \ \ \ \ \ \ \ $%
\begin{equation*}
\ \ 2i19^{k}=\left( u+iv\right) ^{3}-\left( u-iv\right) ^{3}
\end{equation*}

or \ \ \ \ \ \ \ \ \ \ \ \ \ \ 

\ \ \ \ \ \ \ \ \ \ \ \ \ \ \ \ \ \ \ \ \ \ \ \ \ \ \ \ \ \ \ \ \ \ \ \ \ \ 
\begin{equation*}
19^{k}=v\left( 3u^{2}-v^{2}\right) .
\end{equation*}

Note that $u$ and $v\ $are coprime, otherwise\ any prime factor of $u$ and $%
v $ will also divide both $x$ and $y$. Therefore $v=\pm 1$ \ or $v=\pm
19^{k} $ \ which leads to

$\ \ \ \ \ \ \ \ \ \ \ \ \ \ \ \ \ \ \ \ 19^{k}=\pm 1\left( 3u^{2}-1\right) $
$\ \ \ \ \ $or\ $\ \ \ \ \ \ \ 19^{k}=\pm 19^{k}\left( 3u^{2}-19^{2k}\right) 
$

$\ \ \ \ \ \ \ \ \ \ \ \ \ \ \ \ \ \ \ \ \pm 19^{k}=3u^{2}-1^{\text{\ \ \ \
\ \ \ \ \ \ \ \ \ \ \ \ \ \ \ \ \ \ \ \ \ \ }}1=\pm 1\left(
3u^{2}-19^{2k}\right) $

$\ \ \ \ \ \ \ \ \ \ \ \ \ \ \ \ \ \ \ \ 3u^{2}=1\pm 19^{k}$ $\ \ \ \ \ \ \
\ \ \ \ \ \ \ \ \ \ \ \ \ \ \ \ 3u^{2}=\pm 1+19^{2k}$

Now we consider the first equation, if we consider the sign is negative,
then the right hand side gives negative value, so this is impossible. If we
think the sign is $+$\ and $k$ is an even, then the right hand side is
congruent to $2$ $modulo$ $3$ while the left hand side is divisible by $3$
which is a contradiction. Finally if we consider  the sign is \ $+$ \ and $k$
is odd, then we reach similar contradiction. So both of the two cases have
no solution.

The second equation the sign must be $-1.$ Thus $\left( 19^{k}\right)
^{2}-3u^{2}=1$ \ i.e. $X^{2}-3Y^{2}=1$ \ has a fundamental solution $\left(
X_{1},Y_{1}\right) =\left( 2,1\right) .$ Furthermore $X_{2}=7,\ X_{3}=26...$
\ . \ $\left( X_{m}\right) \ $\ with the recurrence sequence $%
X_{m}=4X_{n-2}-X_{n-1}$ is a Lucas Sequence of \ second type. By Primitive
Divisor Theorem $[4]$ for $m>12\ $and checking $X_{m}$\ for all \ $\leq 12,$
we get $X_{m}$ cannot be power of $19$.

\bigskip

Now we assume $19\mid x.$ Let \ $x=19^{u}.X,$ $\ y=19^{v}.Y$ \ \ where \ $%
u>0,$ $\ v>0$ \ \ and \ \ $\left( 19,X\right) =\left( 19,Y\right) =1.$ Then
equation becomes

$\ \ \ \ \ \ \ \ \ \ \ \ \ \ \ $%
\begin{equation*}
\ \ \left( 19^{u}X\right) ^{2}+19^{2k}=19^{3v}Y^{3}\text{ \ \ \ \ \ \ \ \ \ }%
(2.1)
\end{equation*}

and we have three posibilities.

1) \ $\ 2u=min\left( 2u,3v,2k\right) .$ By cancelling $19^{2u}$ \ in $\left(
2.1\right) $ we get

$\ \ \ \ \ \ \ \ \ \ \ \ \ $%
\begin{equation*}
\ \ X^{2}+19^{2\left( k-u\right) }=19^{3v-2u}Y^{3}.
\end{equation*}

If \ $k-u=0$ \ \ and $\ \ 3v-2u=0,$ \ we get the famous equation of Lebesque
which has no solution $\left[ 10\right] .$

If \ $k-u>0$ \ and \ \ $3v-2u=0,$ \ then considering $modulo$ $19,$ we get
the equation in Claim 4  with the first assume. We know with this assume \
the equation $x^{2}+19^{2k}=y^{3}$ has no solution.

Finally if  $k-u=0$ , we get

$\ \ \ \ \ \ \ \ \ \ \ \ \ \ $%
\begin{equation*}
\ \ X^{2}+1=19^{3v-2k}Y^{3}\text{ \ \ \ \ \ \ \ \ \ }\left( 2.2\right) .
\end{equation*}

If \ $3\mid k$ \ then we can write the equation $\left( 2.2\right) $ as

$\ \ \ \ \ \ \ \ \ \ \ \ \ \ $%
\begin{equation*}
\ \ X^{2}+1=(19^{v-\frac{2k}{3}}Y)^{p}
\end{equation*}

which has no solution.

2) \ $2k=min\left( 2u,3v,2k\right) .\ $By cancelling $19^{2k}$ in $\left(
2.1\right) $ \ we get

$\ \ \ \ \ \ \ \ \ \ \ \ \ $%
\begin{equation*}
\ \ \ \left( 19^{u-k}X\right) ^{2}+1=19^{3v-2k}Y^{3}.
\end{equation*}

Considering this equation $modulo$ $19$ \ we get either $3v-2k=0$ \ i.e.$\
\left( 19^{u-k}X\right) ^{2}+1=Y^{3}$ which has no solution from $Lemma$ $1$
\ or \ $3v-2k>0$ \ namely we get the same equation in $\left( 2.2\right) .$

3) \ $3v=min\left( 2u,3v,2k\right) $

This case do not give us any solution. This completes the proof.

\bigskip

\textbf{The last case is, }$n=p\geq 5$\textbf{\ where }$p$\textbf{\ is a
prime.}

\bigskip

\begin{claim}
The equation $x^{2}+19^{2k}=y^{p},\ k>0$ \ has no positive integer solution.
\end{claim}

Fistly we assume that $(x,19)=1.$ \bigskip A. Berczes and I. Pink $\left[ 3%
\right] $ proved that there are no solution to the equation $\
x^{2}+p^{2k}=y^{n}$, with $p\in \left\{ 19,41,59,61,79\right\} $, $n\geq 5$
and $k\geq 3,\ (x,y)=1.$

Now we assume $19\mid x.$ Let \ $x=19^{u}.X,$ $\ y=19^{v}.Y$ \ \ where \ $%
u>0,$ $\ v>0$ \ \ and \ \ $\left( 19,X\right) =\left( 19,Y\right) =1.$ Then
equation becomes

$\ \ \ \ \ \ \ \ \ \ \ \ \ \ \ \ \ $%
\begin{equation*}
\left( 19^{u}X\right) ^{2}+19^{2k}=19^{pv}Y^{p}\text{ \ \ \ \ \ \ \ \ \ }%
(2.3)
\end{equation*}

and we have three posibilities.

1) \ $\ 2u=min\left( 2u,pv,2k\right) .$ By cancelling $19^{2u}$ \ in $\left(
2.3\right) $ we get

$\ \ \ \ \ \ \ \ \ \ \ \ $%
\begin{equation*}
\ \ \ X^{2}+19^{2\left( k-u\right) }=19^{pv-2u}Y^{p}.
\end{equation*}

If \ $k-u=0$ \ \ and $\ \ pv-2u=0,$ \ we get the famous equation of Lebesque
which has no solution $\left[ 10\right] .$

If \ $k-u>0$ \ and \ \ $pv-2u=0,$ \ then considering $modulo$ $19,$ we get
Berczes and Pink's equation $\left[ 3\right] \ $. We know this equation has
no solution.

Finally if $k-u=0$ , we get

$\ \ \ \ \ \ \ \ \ \ \ \ \ \ \ $%
\begin{equation*}
\ X^{2}+1=19^{pv-2k}Y^{p}\text{ \ \ \ \ \ \ \ \ \ }\left( 2.4\right) .
\end{equation*}

If \ $p\mid k$ \ then we can write the equation $\left( 2.4\right) $ as

$\ \ \ \ \ \ \ \ \ \ \ \ $%
\begin{equation*}
\ \ \ \ X^{2}+1=(19^{v-\frac{2k}{p}}Y)^{p}
\end{equation*}

which has no solution.

2) \ $2k=min\left( 2u,pv,2k\right) .\ $By cancelling $19^{2k}$ in $\left(
2.3\right) $ \ we get

$\ \ \ \ \ \ \ \ \ \ \ \ $%
\begin{equation*}
\ \ \ \ \left( 19^{u-k}X\right) ^{2}+1=19^{pv-2k}Y^{p}.
\end{equation*}

Considering this equation $modulo$ $19$ \ we get either $pv-2k=0$ \ i.e.$\
\left( 19^{u-k}X\right) ^{2}+1=Y^{p}$ which has no solution from $Lemma$ $1$
\ or \ $pv-2u>0$ \ namely we get the same equation in $\left( 2.4\right) .$

3) \ $pv=min\left( 2u,pv,2k\right) $

This case do not give us any solutions. 

All of the cases complete the proof of the Theorem 2.

\bigskip 

\bigskip \textbf{3. The equation }$x^{2}+19^{2k+1}=y^{n}$

\bigskip For the equation $x^{2}+19^{2k+1}=y^{n}$ the case $k=0$, $n\geq 3$
has been studied by Cohn $\left[ 6\right] \ $and the case $k\geq 0$, $n\geq 5
$ has been studied by Arif and Muriefah $\left[ 2\right] $. Therefore  we
consider the case $k>0$ and $n=3,4.$

\bigskip

Now, we can give our main theorem for $m$ is an odd.

\bigskip

\begin{theorem}
The equation $\ x^{2}+19^{2k+1}=y^{n}$, \textit{where }$n=3,4$\textit{\ and }%
$k>0$\textit{\ has no positive integer solution.}
\end{theorem}

\begin{proof}
For proof we consider the cases $n=3$ and $n=4$ separately.
\end{proof}

\textbf{The case }$n=3$\textbf{, }$k>0$

For this case, we first assume $\left( 19,x\right) =1$. There is no loss
generality in considering only $n=3$\ is an odd prime. Since the class
number of the field $%
\mathbb{Q}
\left( \sqrt{-19}\right) $ is not multiply by $n=3,$ we have \ 

\ \ \ \ \ \ \ \ \ \ \ \ 
\begin{equation*}
\ \ x+19^{k}\sqrt{-19}=\left( \dfrac{a+b\sqrt{-19}}{2}\right) ^{3}
\end{equation*}

where \ $y=\dfrac{a^{2}+19b^{2}}{4}$ \ for some rational integers $a$ \ and $%
\ b$. Equating imaginary parts, we get

$\ \ \ \ \ \ \ \ \ \ \ $%
\begin{eqnarray*}
\ \ 19^{k}.2^{3} &=&b.\left[ \binom{3}{1}a^{2}+\binom{3}{3}\left(
-19b^{2}\right) \right]  \\
19^{k}.8 &=&b.\left( 3a^{2}-19b^{2}\right) .
\end{eqnarray*}

For this equation we have three possibilities. If we take $b=\pm 1,$ then we
obtain

\ \ \ \ \ \ \ \ \ \ \ \ 
\begin{equation*}
\pm 19^{k}.8=3a^{2}-19.
\end{equation*}

There is no solution of this equation.

If we take $b=\pm 19^{k}$ then we get

\ \ \ \ \ \ \ \ \ \ 
\begin{equation*}
\ \ \pm 8=3a^{2}-19^{2k+1}\text{ \ \ \ \ \ \ \ \ \ }(3.1)
\end{equation*}

where $k>0$. There is no solution of this equation too.

If we take $b=\pm 19^{\lambda }\ (0<\lambda <k)$ then we get

\ \ \ \ \ \ \ \ \ \ \ 
\begin{equation*}
\ \pm 19^{k-\lambda }.8=3a^{2}-19^{2\lambda +1}.
\end{equation*}

If $k-\lambda >0$, this is not possible $modulo\ 19$. If $k-\lambda =0$ that
is $k=\lambda ,$ then we get the equation $(3.1)$ again and we know this
equation doesn't have a solution, \ where $k>0$. So there is no solution of
this equation.

\bigskip\ \ \ \ \ \ \ 

\ Secondly we assume $19\mid x.$ Let $x=19^{s}.X,$ $\ y=19^{t}.Y$ \ where \ $%
s>0,\ t>0$ \ and\ $\left( 19,X\right) =\left( 19,Y\right) =1.$ Then equation
becomes

$\ \ \ \ \ \ \ \ \ \ \ \ \ \ $%
\begin{equation*}
\ \ \ \left( 19^{s}X\right) ^{2}+19^{2k+1}=19^{3t}Y^{3}\text{ \ \ \ \ \ \ \
\ \ }(3.2)
\end{equation*}

and we have three posibilities.

1) \ $\ 2s=min\left( 2s,3t,2k+1\right) .$ By cancelling $19^{2s}$ \ in $%
\left( 3.2\right) $ we get

$\ \ \ \ \ \ \ \ \ \ \ \ \ $%
\begin{equation*}
\ \ X^{2}+19^{2\left( k-s\right) +1}=19^{3t-2s}Y^{3}
\end{equation*}

and considering this equation $modulo$ $19.\ $ \ 

$\ \ \ \ \ \ \ \ \ \ \ \ \ \ $%
\begin{equation*}
\ X^{2}+19^{2\left( k-s\right) +1}=Y^{3}
\end{equation*}

We deduce that $\ 3t-2s=0,$ i.e. $3t=2s$ then $3\mid s$, $s=3M$.$\ $For $%
k-s>0,$\ this equation has no solution.

2) \ $2k+1=min\left( 2s,3t,2k+1\right) .$ Then \ we get

$\ \ \ \ \ \ \ \ \ \ \ \ \ \ $%
\begin{equation*}
\ \ 19^{2s-2k-1}X^{2}+1=19^{3t-2k-1}Y^{3}
\end{equation*}

and considering this equation $modulo$ $19,$\ we get $3t-2k-1=0,$\ so that

$\ \ \ \ \ \ \ \ \ \ \ \ \ \ \ $%
\begin{equation*}
\ \ 19\ \left( 19^{s-k-1}X\right) ^{2}+1=Y^{3}
\end{equation*}

we obtain

\ $\ \ \ \ \ \ \ \ \ \ \ \ \ \ \ \ $%
\begin{equation*}
\ 19Z^{2}+1=Y^{3}
\end{equation*}

from $Lemma$ $1$ we can say this equation has no positive integer solution.

3) \ $3t=min\left( 2s,3t,2k+1\right) $

then we get

$\ \ \ \ \ \ \ \ \ \ \ \ \ \ \ $%
\begin{equation*}
\ 19^{2s-3t}X^{2}+19^{2k+1-3t}=Y^{3}
\end{equation*}

this is impossible$\ modulo$ $19$ only if $2s-3t=0$ \ or $2k+1-3t=0$\ and
both of these cases have already been discussed. This includes the proof of
the theorem's first case.

\ 

\bigskip 

\textbf{The case }$n=4$\textbf{, }$k>0$

\bigskip When $n=4$ \ we have $x^{2}+19^{2k+1}=y^{4}$. \ If \ $y$ \ is an
even, then $y^{4}\equiv 0\ \left( \func{mod}8\right) $. If we reduce the
equation for $modulo\ 8$, then we get $x^{2}+3\equiv 0$ $\left( \func{mod}%
8\right) $ namely $x^{2}\equiv 5$ $\left( \func{mod}8\right) .$ From
Legendre symbol $\left( \frac{5}{8}\right) =-1,$ this equation is not
solvable. So $y$ is an odd and $x$ is an even due to $\left( x,y\right) =1$.
We have$\ 19^{2k+1}=\left( y^{4}-x^{2}\right) =\left( y^{2}-x\right) \left(
y^{2}+x\right) .$Then

\ \ \ \ \ \ \ \ \ \ \ \ \ \ \ \ \ \ \ \ \ \ \ \ \ \ \ \ \ \ \ \ 
\begin{equation*}
\ y^{2}-x=1
\end{equation*}

$\ \ \ \ \ \ \ \ \ \ \ \ \ \ \ \ \ \ \ \ \ \ \ \ \ \ \ \ \ \ \ \ $%
\begin{equation*}
\ y^{2}+x=19^{2k+1}
\end{equation*}

so

\ \ \ \ \ \ \ \ \ \ \ \ \ \ \ \ \ \ \ \ \ \ \ \ \ \ \ \ \ \ \ 
\begin{equation*}
\ \ 2y^{2}=19^{2k+1}+1.
\end{equation*}

Then $2y^{2}\equiv 4\ \left( \func{mod}8\right) $ i.e. $y^{2}\equiv 2\
\left( \func{mod}4\right) $\ which is impossible. 

This completes proof of the Theorem 6.

\textbf{\bigskip }

\end{document}